\documentclass{amsart}

\usepackage{hyperref}
\usepackage{enumerate}
\usepackage{amssymb}
\usepackage{graphicx}
\usepackage{MnSymbol}

\usepackage{amsthm}

\title{J\'onsson groups of various cardinalities}

\author{Samuel M. Corson}
\address{School of Mathematics, University of Bristol, Fry Building, Woodland Road, Bristol, BS8 1UG, United Kingdom.}
\email{sammyc973@gmail.com}

\theoremstyle{definition}\newtheorem{theorem}{Theorem}

\theoremstyle{definition}

\theoremstyle{definition}\newtheorem{corollary}{Corollary}
\theoremstyle{definition}\newtheorem{proposition}[theorem]{Proposition}
\theoremstyle{definition}\newtheorem{definition}[theorem]{Definition}
\theoremstyle{definition}
\theoremstyle{definition}
\theoremstyle{definition}
\theoremstyle{definition}
\theoremstyle{definition}\newtheorem{lemma}[theorem]{Lemma}
\theoremstyle{definition}
\theoremstyle{definition}
\theoremstyle{definition}\newtheorem{definitions}[theorem]{Definitions}
\theoremstyle{definition}
\theoremstyle{definition}
\theoremstyle{definition}

\newcommand{\Po}{\mathcal{P}}
\newcommand{\Aut}{\operatorname{Aut}}
\newcommand{\Pof}{\mathcal{P}_{\text{fin}}}

\begin{document}

\thanks{This work was supported by the Heilbronn Institute for Mathematical Research, Bristol, UK.}

\begin{abstract}
A group $G$ is J\'onsson if $|H| < |G|$ whenever $H$ is a proper subgroup of $G$.  Using an embedding theorem of Obraztsov it is shown that there exists a J\'onsson group $G$ of infinite cardinality $\kappa$ if and only if there exists a J\'onsson algebra of cardinality $\kappa$.  Thus the question as to which cardinals admit a J\'onsson group is wholly reduced to the well-studied question of which cardinals are not J\'onsson.  As a consequence there exist J\'onsson groups of arbitrarily large cardinality.  Another consequence is that the infinitary edge-orbit conjecture of Babai is true.

\end{abstract}

\maketitle

\begin{section}{Introduction}

A group $G$ is \emph{J\'onsson} (named for Bjarni J\'onsson) if every proper subgroup of $G$ has cardinality strictly less than that of $G$.  Finite groups are J\'onsson and countably infinite examples include the quasi-cyclic groups $\mathbb{Z}(p^{\infty})$ and the Tarski monsters constructed by Ol'shanskii (\cite{Ol0}, \cite[Theorem 28.1]{Ol}).  A J\'onsson group of cardinality $\aleph_1$ was constructed by Shelah in \cite{Sh} and further examples of J\'onsson groups of cardinality $\aleph_1$ with striking properties were later obtained by Obraztsov (\cite{Ob000}, \cite[Corollary 35.4]{Ol}).

There are known to exist J\'onsson groups of cardinality $\aleph_2, \aleph_3, \ldots$ (see \cite[Theorem E]{MoOb}) but there do not appear to be any of cardinality $\aleph_{\omega}$ or higher in the literature which are constructed only from the standard ZFC axioms.  It is known that if there is a J\'onsson group of cardinality $\kappa$ then there exists one of cardinaity $\kappa^+$ \cite{MoOb}, where $\kappa^+$ denotes the successor cardinal to $\kappa$ (the smallest cardinal which is strictly greater than $\kappa$).  If  $2^{\lambda} = \lambda^+$ then there exists a J\'onsson group of cardinality $\lambda^+$  \cite{Sh}, but the failure of this equality at every cardinal $\lambda > 1$ is consistent with ZFC, modulo some large cardinal assumptions \cite{ForWoo}.

Before stating the main theorem we recall some definitions.  Recall that an \emph{algebra} is an ordered pair $(A, \mathcal{F})$ where $A$ is a set and $\mathcal{F}$ is a collection of finitary operations on $A$.  An algebra $(A, \mathcal{F})$ is \emph{J\'onsson} if $\mathcal{F}$ is countable and any proper subalgebra is of cardinality strictly less than $|A|$.  The main result is the following.

\begin{theorem}\label{mainJonssontheorem}  For $\kappa$ an uncountable cardinal, there exists a J\'onsson algebra of cardinality $\kappa$ if and only if there exists a simple, torsion-free J\'onsson group of cardinality $\kappa$.
\end{theorem}

Naturally a J\'onsson group is a J\'onsson algebra, so we are only concerned with producing a J\'onsson group given the existence of a J\'onsson algebra.  From Theorem \ref{mainJonssontheorem} one obtains many new examples of J\'onsson groups.

\begin{corollary}\label{cutefacts}  There is a J\'onsson group of cardinality $\kappa$ when
\begin{enumerate}

\item $\kappa$ is the successor of a regular cardinal \cite{Try};

\item $\kappa = \lambda^+$, where $\lambda$ is singular and not a limit of weakly inaccessible cardinals \cite{Sh2}; or

\item $\kappa = \beth_{\omega}^+$ \cite{Sh3}.

\end{enumerate}

\end{corollary}

For example, there exist J\'onsson groups of cardinality $\aleph_1$, $\aleph_2$, $\aleph_{\omega + 1}$, $\aleph_{\omega^2 + 7}$, and $\aleph_{\omega_1 + 1}$.  Since the successor of a cardinal is regular, we obtain from (1) that for each infinite cardinal $\kappa$ there exists a J\'onsson group of cardinality $\kappa^{++}$.  Thus there are J\'onsson groups of arbitrarily large (regular) cardinality and this in itself has a consequence in combinatorics which we mention.

The infinitary edge-orbit conjecture of L. Babai \cite[Conjecture 5.22]{Bab} is the following: 

\begin{center}  \emph{For each cardinal $\kappa$ there exists a group $G$ such that whenever $G \simeq \Aut(\Gamma)$ for some graph $\Gamma$ the action of $\Aut(\Gamma)$ on the edges of $\Gamma$ has at least $\kappa$ orbits.}

\end{center}

\noindent  It was shown by Simon Thomas that this conjecture is true provided there exist J\'onsson groups of arbitrarily large regular cardinality \cite{Tho} and so the following is immediate.

\begin{corollary}  The infinitary edge-orbit conjecture is true.
\end{corollary}

For another application, one can combine Theorem \ref{mainJonssontheorem} with \cite[Theorem G]{MoOb} to see that for any infinite cardinal $\kappa$ there exists a topological group $H$ of cardinality $\kappa^{+++}$ which is not discrete, with $H$ simple and J\'onsson and every proper subgroup of $H$ is discrete.  Keisler and Rowbottom have shown that there are J\'onsson algebras of every infinite cardinality in G\"odel's constructible universe $L$ \cite[Corollary 9.1]{Dev} and so in the constructible universe there exist J\'onsson groups of every nonzero cardinality.

We caution the reader that there are some results about J\'onsson algebras which do not carry over into the area of J\'onsson groups.  Recall that an algebra is \emph{locally finite} if each finitely generated subalgebra is finite.  It is known that there exists a J\'onsson algebra of cardinality $\kappa$ if and only if there exists a locally finite J\'onsson algebra of cardinality $\kappa$ \cite[Theorem 3.10]{Dev}.  On the other hand if $\kappa$ is an uncountable regular cardinal then there cannot exist a locally finite J\'onsson group of cardinality $\kappa$ \cite[Theorem 2.6]{KegWehr}.

\end{section}

\begin{section}{The proof}

The proof of Theorem \ref{mainJonssontheorem} will make use of a classical result regarding J\'onsson algebras as well as a group embedding theorem.  The following is obtained from a result of \L{}os (see \cite[Theorem 3.4]{Dev}).

\begin{lemma}\label{Los}  There is a J\'onsson algebra on an infinite set $X$ if and only if there exists a binary operation $j: X \times X \rightarrow X$ such that $(X, \{j\})$ is J\'onsson.
\end{lemma}

We will prepare to state a group embedding theorem by recalling some definitions.  For $X$ a set we let $\Po(X)$ denote the powerset of $X$ and $\Pof(X)$ denote the collection of finite subsets of $X$.

\begin{definitions}  Given a collection of nontrivial groups without involutions $\{G_i\}_{i \in I}$ we call the set $\bigcup_{i \in  I} G_i$, where $\{1\} = G_i \cap G_j$ whenever $i \neq j$, the \emph{free amalgam} of the groups $\{G_i\}_{i \in I}$ and will denote it by $\Omega^1$.  We will let $\Omega \subseteq \Omega^1$ denote the subset $\Omega^1 \setminus \{1\}$.  An \emph{embedding} of $\Omega^1$ to a group $G$ is an injective function $E: \Omega^1 \rightarrow G$ where each restriction $E \upharpoonright G_i$ is a homomorphism.  An embedding $E$ extends to a homomorphism $\phi$ from the free product $*_{i \in I} G_i$ to $G$, and  since each $\phi \upharpoonright G_i$ is an isomorphism, we consider each $G_i$ to be a subgroup of $G$ in such a situation.
\end{definitions}

\begin{definition}\label{generating}  Let $\{G_i\}_{i \in I}$ be a collection of nontrivial groups without involutions, $\Omega^1$ be the free amalgam, and $\Omega = \Omega^1 \setminus \{1\}$ as above.  A function $f: \Po(\Omega) \setminus \{\emptyset\} \rightarrow \Po(\Omega)$ is \emph{generating} provided 

\begin{enumerate}[(a)]

\item if $X \subseteq G_i$ for some $i \in I$ then $f(X) = \langle X \rangle \setminus \{1\}$;

\item if $X \subseteq \Omega$ is finite and $X \not\subseteq G_i$ for all $i \in I$ then $f(X) = Y$ where $Y$ is a countable subset of $\Omega$ such that $X \subseteq Y$ and if $Z \subseteq Y$ is finite nonempty then $f(Z) \subseteq Y$;

\item if $X \subseteq \Omega$ is infinite and $X \not\subseteq G_i$ for all $i \in I$ then $f(X) = \bigcup_{Z \in \Pof(X) \setminus \{\emptyset\}} f(Z)$ .

\end{enumerate}

\end{definition}

The following is a special case of a beautiful embedding theorem of Obraztsov \cite[Theorem A]{Ob0}.

\begin{proposition}\label{embedding}  Suppose that $\{G_i\}_{i \in I}$ is a collection of nontrivial groups without involutions and that $f$ is a generating function on this collection, with $|I| \geq 2$.  Then there is a simple group $G$ and an embedding $E: \Omega_1 \rightarrow G$ of the free amalgam which induces a homomorphism $\phi: *_{i \in I} G_i \rightarrow G$ satisfying the following properties:

\begin{enumerate}

\item $\phi$ is surjective;

\item if $h \in G$ is not conjugate in $G$ to an element in one of the groups $G_i$ then $h$ is of infinite order;

\item each subgroup $M$ of $G$ is either cyclic, or conjugate in $G$ to a subgroup of one of the $G_i$, or conjugate in $G$ to a subgroup of form $\langle C \rangle$ where $C = f(X)$ for some $\emptyset \neq X \subseteq \Omega$.
\end{enumerate}

\end{proposition}

\begin{proof}[Proof of Theorem \ref{mainJonssontheorem}]  Suppose that there is a J\'onsson algebra of uncountable cardinality $\kappa$.  Let $\{G_{\alpha}\}_{\alpha < \kappa}$ be a collection of groups where each $G_{\alpha}$ is infinite cyclic and generated by $z_{\alpha}$.  Let $\Omega^1$ be the free amalgam of the groups $\{G_{\alpha}\}_{\alpha < \kappa}$ and $\Omega = \Omega^1 \setminus \{1\}$.  Let $\tau: \Omega \rightarrow \{z_{\alpha}\}_{\alpha < \kappa}$ be given by $g \mapsto z_{\alpha}$ where $g \in G_{\alpha}$.

As there is a J\'onsson algebra of cardinality $\kappa$, we have by Lemma \ref{Los} a binary operation $j: \{z_{\alpha}\}_{\alpha < \kappa} \times \{z_{\alpha}\}_{\alpha < \kappa} \rightarrow \{z_{\alpha}\}_{\alpha < \kappa}$ such that $(\{z_{\alpha}\}_{\alpha < \kappa}, \{j\})$ is a J\'onsson algebra.  For each subset $Y \subseteq \{z_{\alpha}\}_{\alpha < \kappa}$ we let $\overline{j}(Y) \subseteq \{z_{\alpha}\}_{\alpha < \kappa}$ denote the subalgebra generated by $Y$ under the operation $j$.  Let $\pi: \{z_{\alpha}\}_{\alpha < \kappa} \rightarrow \kappa$ be given by $z_{\alpha} \mapsto \alpha$.

We claim that the function $f: \Po(\Omega) \setminus \{\emptyset\} \rightarrow \Po(\Omega)$ defined by

\[
f(X) = \left\{
\begin{array}{ll}
\langle X \rangle \setminus \{1\}
                                            & \text{if } X \subseteq G_{\alpha}\text{ for some }\alpha \in \kappa, \\
(\bigcup_{\alpha \in \pi(\overline{j}(\tau(X)))} G_{\alpha}) \setminus \{1\}                                           & \text{otherwise}
\end{array}
\right.
\]

\noindent is generating.  Certainly condition (a) of Definition \ref{generating} holds.  To check condition (b) we let $X \subseteq \Omega$ be finite with $X \not\subseteq G_{\alpha}$ for all $\alpha < \kappa$.  The set $\tau(X) \subseteq \{z_{\alpha}\}_{\alpha < \kappa}$ is finite, and so the set $\overline{j}(\tau(X))$ generated by $\tau(X)$ under $j$ will be countable.  Then $\pi(\overline{j}(\tau(X)))$ is countable.  Thus the set $f(X) = (\bigcup_{\alpha \in \pi(\overline{j}(\tau(X)))} G_{\alpha}) \setminus \{1\}$ is countable as a countable union of countable sets.  Certainly $f(X) \supseteq X$ in this case, since 

\begin{center}
$X \subseteq (\bigcup_{\alpha \in \pi(\tau(X))} G_{\alpha}) \setminus \{1\} \subseteq (\bigcup_{\alpha \in \pi(\overline{j}(\tau(X)))} G_{\alpha}) \setminus \{1\}  = f(X)$.
\end{center}

\noindent Moreover given a finite nonempty $Z \subseteq f(X)$ we either have $Z \subseteq G_{\alpha}$ for some $\alpha < \kappa$, in which case $\alpha \in \pi(\overline{j}(\tau(X)))$ and 

\begin{center}
$f(Z) = \langle Z \rangle \setminus \{1\} \subseteq G_{\alpha} \setminus \{1\} \subseteq f(X)$
\end{center}

\noindent or else we have

\begin{center}
$f(Z) = (\bigcup_{\alpha \in \pi(\overline{j}(\tau(Z)))} G_{\alpha}) \setminus \{1\} \subseteq  (\bigcup_{\alpha \in \pi(\overline{j}(\tau(X)))} G_{\alpha}) \setminus \{1\}  = f(X)$
\end{center}

\noindent since $\tau(Z) \subseteq \overline{j}(\tau(X))$.  Thus condition (b) holds.  For condition (c) we let $X \subseteq \Omega$ be infinite such that $X \not\subseteq G_{\alpha}$ for each $\alpha \in \kappa$.  Given $g \in f(X)$ we have $g \in G_{\alpha} \setminus \{1\}$ for some $\alpha \in \pi(\overline{j}(\tau(X)))$.  Thus $z_{\alpha} \in \overline{j}(\tau(X))$ and we may select $Z_0 \subseteq \tau(X)$ which is finite such that $z_{\alpha} \in \overline{j}(Z_0)$, and we may assume without loss of generality that $Z_0$ has at least two elements (since $\tau(X)$ has at least two elements).  Select a finite $Z \subseteq X$ such that $\tau(Z) = Z_0$ and it is clear that $g \in f(Z)$.  Therefore $f(X) \subseteq \bigcup_{Z \in \Pof(X) \setminus \{\emptyset\}} f(Z)$.  To see the inclusion $f(X) \supseteq \bigcup_{Z \in \Pof(X) \setminus \{\emptyset\}}f(Z)$ it is sufficient to show $f(X) \supseteq f(Z)$ for every $Z \in \Pof(X) \setminus \{\emptyset\}$.  This is almost the same check as in part (b): either $Z \subseteq G_{\alpha}$ for some $\alpha < \kappa$, so that $\alpha \in \pi(\tau(X)) \subseteq \pi(\overline{j}(\tau(X)))$ and 

\begin{center}
$f(Z) = \langle Z \rangle \setminus \{1\} \subseteq G_{\alpha} \setminus \{1\} \subseteq f(X)$
\end{center}

\noindent or else

\begin{center}
$f(Z) = (\bigcup_{\alpha \in \pi(\overline{j}(\tau(Z)))} G_{\alpha}) \setminus \{1\} \subseteq  (\bigcup_{\alpha \in \pi(\overline{j}(\tau(X)))} G_{\alpha}) \setminus \{1\}  = f(X)$.
\end{center}

\noindent Thus $f(X) = \bigcup_{Z \in \Pof(X) \setminus \{\emptyset\}} f(Z)$ and (c) holds.

We may therefore apply Proposition \ref{embedding} to produce a simple group $G$ into which the free amalgam $\Omega^1$ embeds and satisfying the listed properties of Proposition \ref{embedding}.  Notice that $G$ is torsion-free by part (2) of  Proposition \ref{embedding}, since each of the subgroups $G_{\alpha}$ is torsion-free.  Also, $|G| = \kappa$ since $\Omega^1$ injects into $G$ and $G = \langle \bigcup_{\alpha < \kappa} G_{\alpha} \rangle$ with $|G_{\alpha}| = \aleph_0$ for each $\alpha < \kappa$. 

It remains to see that $G$ is J\'onsson.  Letting $M$ be a subgroup of $G$ with $|M| = \kappa$, we know by part (3) of Proposition \ref{embedding} that up to conjugation $M$ is equal to $\langle C \rangle$ where $C = f(X)$ for a nonempty $X \subseteq \Omega$.  Thus by conjugating if necessary we may assume without loss of generality that $M = \langle C \rangle$ for such a $C$.  Since $|M| = \kappa$ it is clear that $|C| = \kappa$ since $\kappa$ is uncountable.  If $X$ were to satisfy $X \subseteq G_{\alpha}$ for some $\alpha < \kappa$ then $f(X) = C \subseteq G_{\alpha}$, but $C$ is uncountable, a contradiction.  Thus $X$ is not a subset of any $G_{\alpha}$, and $X$ cannot be finite for then $C$ would be countable by condition (b).  Therefore $C = f(X) = \bigcup_{Z \in \Pof(X) \setminus \{\emptyset\}}f(Z)$.  Since $f(Z)$ is countable for any $Z \in \Pof(X) \setminus \{\emptyset\}$, and $|\Pof(X) \setminus \{\emptyset\}| = |X|$ since $X$ is infinite, we see that $\kappa = |C| \leq |X|\cdot \aleph_0$, and therefore $|X| = \kappa$.

Since the function $\tau: \Omega \rightarrow \{z_{\alpha}\}_{\alpha < \kappa}$ is countable-to-one we see that $|\tau(X)| = \kappa$, and as $j$ is as in Lemma \ref{Los} we get $\overline{j}(\tau(X)) = \{z_{\alpha}\}_{\alpha < \kappa}$.  Thus $\pi(\overline{j}(\tau(X))) = \kappa$ and 

\begin{center}

$C = f(X) = (\bigcup_{\alpha \in \pi(\overline{j}(\tau(X)))} G_{\alpha}) \setminus \{1\} = \Omega$

\end{center}

\noindent from which we have $M = \langle \Omega \rangle = G$ and we are done.

\end{proof}

\end{section}

\begin{section}*{Acknowledgement}

The author gives thanks to the anonymous referee for careful reading and suggested improvements of the paper.

\end{section}

\end{document}